\documentclass[11pt]{article}

\usepackage{amsmath,amsthm}

\usepackage{amssymb,latexsym}

\usepackage{enumerate}

\newcommand{\BB}{{\cal B}}

\newcommand{\EE}{{\cal E}}
\newcommand{\FF}{{\cal F}}

\newcommand{\HH}{{\cal H}}

\newcommand{\MM}{{\cal M}}

\newcommand{\RR}{{\cal R}}
\newcommand{\VV}{{\cal V}}
\newcommand{\WW}{{\cal W}}

\newcommand{\BM}{{\mathbb M}}
\newcommand{\BN}{{\mathbb N}}
\newcommand{\BR}{{\mathbb R}}

\newcommand{\BBM}{{\mathbf M}}
\newcommand{\BBX}{{\mathbf X}}

\newcommand{\fch}{{\mathbf{1}}}

\newtheorem{theorem}{\bf Theorem}[section]
\newtheorem{corollary}[theorem]{\bf Corollary}%[subsection]

\theoremstyle{definition}
\newtheorem{definition}[theorem]{Definition}

\newtheorem{remark}[theorem]{Remark}

\numberwithin{equation}{section}

\begin{document}

\title {Renormalized solutions of semilinear equations
involving measure data and operator corresponding to Dirichlet
form}
\author {Tomasz Klimsiak and Andrzej Rozkosz}
\date{}
\maketitle
\begin{abstract}
We generalize the notion of renormalized solution to semilinear
elliptic and parabolic equations involving operator associated
with general (possibly nonlocal) regular Dirichlet form and smooth
measure on the right-hand side. We show that under mild
integrability assumption on the data a quasi-continuous function
$u$ is a renormalized solution to an elliptic (or parabolic)
equation  in the sense of our definition iff $u$ is its
probabilistic solution, i.e.  $u$ can be represented by a suitable
nonlinear Feynman-Kac formula. This implies in particular that for
a broad class of local and nonlocal semilinear equations there
exists a unique renormalized solution.
\end{abstract}

\footnotetext{T. Klimsiak: Institute of Mathematics, Polish
Academy of Sciences, \'Sniadeckich 8, 00-956 Warszawa, Poland, and
Faculty of Mathematics and Computer Science, Nicolaus Copernicus
University, Chopina 12/18, 87-100 Toru\'n, Poland. e-mail:
tomas@mat.umk.pl}

\footnotetext{A. Rozkosz: Faculty of Mathematics and Computer
Science, Nicolaus Copernicus University, Chopina 12/18, 87-100
Toru\'n, Poland. e-mail: rozkosz@mat.umk.pl}

\section{Introduction}

The aim of this paper is to extend  the notion of renormalized
solution to encompass semilinear elliptic and parabolic equations
involving measure data and operators associated with  Dirichlet
forms. The paper consists of two parts. In the first one we are
concerned with elliptic equations of the form
\begin{equation}
\label{eq1.1} -Lu=f(x,u)+\mu.
\end{equation}
In (\ref{eq1.1}), $L$ is the operator associated with a  regular
Dirichlet form $(\EE,D(\EE))$ on $L^2(E;m)$ and
$f:E\times\BR\rightarrow\BR$ is a measurable function. As for
$\mu$ we assume that it is a bounded smooth measure on $E$, i.e. a
measure of bounded total variation on $E$ which charges no set of
zero capacity associated with the form $(\EE,D(\EE))$. Note that
the class of operators $L$ we consider is quite large. It contains
many local as well as nonlocal operators. The model examples are
Laplacian and fractional Laplacian (many other examples are to be
found for instance in \cite{FOT,KR:JFA,KR:CM,MR}).

An important problem one encounters when dealing with equations of
the form (\ref{eq1.1}) is to define properly a solution. In case
$L$ is  local and (\ref{eq1.1}) is linear, i.e. $f$ does not
depend on $u$, some definition, now called Stampacchia's
definition  by duality, was proposed in \cite{Stam}. To deal with
semilinear equations  the definitions of entropy solution (see
\cite{BBGGPV}) and of renormalized solution (see \cite{DMOP}) have
been introduced. For a comparison of different forms of these
definitions as well as remarks on other concepts of solutions see
\cite{DMOP}. In case $E=D\subset\BR^d$ is a bounded domain and $L$
is a uniformly elliptic operator in divergence form with Dirichlet
boundary conditions associated with the classical form
\[
\EE(\varphi,\psi)=\int_D(a\nabla\varphi,\nabla\psi)\,dx,\quad
\varphi,\psi \in H^1_0(D)
\]
one of the equivalent definitions of a solution of (\ref{eq1.1})
given in \cite{DMOP} says that a quasi-continuous
$u:D\rightarrow\BR$ is a renormalized solution of (\ref{eq1.1}) if
$f(\cdot,u)\in L^1(D)$, $T_ku\in H^1_0(D)$ for $k>0$, where
$T_ku=((-k)\vee u)\wedge k$,  and there exists a  sequence
$\{\nu_k\}$ of bounded smooth measures on $D$ such that
$\|\nu_k\|_{TV}\rightarrow0$ as $k\rightarrow\infty$ and for any
bounded quasi-continuous $v\in H^1_0(D)$ and $k\in\BN$,
\begin{equation}
\label{eq1.4} \int_D(a\nabla(T_ku),\nabla v)\,dx
=\int_Df(x,u)v(x)\,dx+\int_Dv(x)\,dx + \int_Dv(x)\,\nu_k(dx).
\end{equation}
In fact, the notion of entropy or renormalized solution can be
applied to deal with more general then (\ref{eq1.1}) equations in
which  $L$ is a Leray-Lions type operator and $\mu$ is not
necessarily smooth.

Another approach to (\ref{eq1.1}), covering both local and
nonlocal operators, have been proposed in \cite{KR:JFA,KR:CM}. In
this probabilistic in nature approach,  a quasi-continuous (with
respect to the form $\EE$) function $u:E\rightarrow\BR$ is a
solution of (\ref{eq1.1}) if the following nonlinear Feynman-Kac
formula
\begin{equation}
\label{eq1.3} u(x)=E_x\Big(\int^{\zeta}_0f(X_t,u(X_t))\,dt
+\int^{\zeta}_0dA^{\mu}_t\Big)
\end{equation}
is satisfied for quasi-every $x\in E$. Here $\BM=(X,P_x)$ is a
Markov process with life time $\zeta$ associated with $\EE$, $E_x$
denotes the expectation with respect to $P_x$ and $A^{\mu}$ is the
additive functional of $\BM$ associated with $\mu$ in the Revuz
sense (see Section \ref{sec2}). In (\ref{eq1.3}) we only assume
that $u$ is quasi-continuous and the integrals make sense. In
fact, if (\ref{eq1.3}) holds and $f(\cdot,u)\in L^1(E;m)$ then
using the probabilistic potential theory and the theory of
Dirichlet forms one can show that $u$ has some additional
regularity properties. Namely,  $T_ku$ belongs to the extended
Dirichlet space $D_e(\EE)$ for every $k>0$.

In \cite{KR:JFA,KR:CM}  also a purely analytical definition of a
solution of (\ref{eq1.1}) resembling Stampacchia's definition is
proposed (see also \cite{KPU,LPPS} for another approach in case of
linear equation with $L$ being a fractional Laplacian). We call it
a solution in the sense of duality. In \cite{KR:JFA,KR:CM} it is
shown that under quite general assumptions on $\EE,f,\mu$ a
function $u$ is a solution in the sense of duality if and only if
it is a probabilistic solution defined by (\ref{eq1.3}). However,
the definition in the sense of duality seems to be not
particularly handy tool for investigating (\ref{eq1.1}).

The natural  question arises whether the concept of renormalized
solution can be carried over to general (possibly nonlocal)
operators corresponding to $\EE$ (for some partial results in this
direction see \cite{AAB}). An obvious related question to ask is
what is the relation between (\ref{eq1.4}) and (\ref{eq1.3}), i.e.
between renormalized and probabilistic solutions? It appears that
(\ref{eq1.4}) is the right form of the definition to be
generalized to encompass wider class of operators. In the paper,
under the assumption that $\EE$ is transient, we define
renormalized solution of (\ref{eq1.1}) as a quasi-continuous
function $u:E\rightarrow\BR$ such that $f(\cdot,u)\in L^1(E;m)$,
$T_ku\in D_e(\EE)$ for $k>0$ and there is a sequence $\{\nu_k\}$
of bounded smooth measures on $E$ such that
$\|\nu_k\|_{TV}\rightarrow0$ and
\begin{equation}
\label{eq1.5}
\EE(T_ku,v)=\int_Ef(x,u)v(x)\,m(dx)+\int_Ev(x)\,\mu(dx) +
\int_Ev(x)\,\nu_k(dx)
\end{equation}
for every $k\in\BN$ and every bounded quasi-continuous $v\in
D_e(\EE)$.  Thus (\ref{eq1.5}) is a direct extension of
(\ref{eq1.4}) to general transitive Dirichlet forms. Our main
theorem says that for transitive forms (\ref{eq1.3}) is equivalent
to (\ref{eq1.5}), or more precisely, that  $u$ is a probabilistic
solution of (\ref{eq1.1}) if and only if it is a renormalized
solution of (\ref{eq1.1}). Since one can prove that under some
assumptions on $f$ there exists a unique probabilistic solution of
(\ref{eq1.1}) for $L$ associated with $\EE$ (see
\cite{KR:JFA,KR:CM} and Section \ref{sec3} for some examples), our
result a fortiori says that (\ref{eq1.5}) provides right
definition of a solution. In particular,  (\ref{eq1.5}) ensures
uniqueness  for interesting classes of equations. In general, the
equivalence of (\ref{eq1.3}) and (\ref{eq1.5}) sheds new light on
the nature of both probabilistic and analytic (renormalized)
solutions of (\ref{eq1.1}). What is perhaps more important, it
also says that in the study of (\ref{eq1.1}) one can use both
probabilistic and analytical methods from the theory of PDEs. Let
us point out once again, that contrary to \cite{DMOP}, in our
theorem we assume that the measure $\mu$ is smooth. An interesting
open problem is how to define renormalized solutions for general
bounded measures, at least for some classes of nonlocal operators.
Finally, let us note that in case $L=\Delta$ the equivalence
between probabilistic and renormalized solutions to (\ref{eq1.1})
was observed in \cite{KR:JEE}.

In the second part of the paper we consider parabolic equation of
the form
\begin{equation}
\label{eq1.2} -\frac{\partial u}{\partial t}-L_tu
=f(t,x,u)+\mu,\quad u(T)=\varphi,
\end{equation}
where $\varphi:E\rightarrow\BR,$ $f:[0,T]\times E\rightarrow\BR$,
the operators $\frac{\partial}{\partial t}+L_t$ correspond to some
time dependent regular Dirichlet form $\EE^{0,T}$ and $\mu$ is a
bounded measure on $(0,T]\times E$ which is smooth with respect to
the capacity associated with $\EE^{0,T}$.

In case $L_t$ are local, a definition of a renormalized  solution
of equations of the form (\ref{eq1.2}) involving more general
nonlinear local operators $L_t$ of Leray-Lions type but with $f$
not depending on $u$ have been introduced in \cite{DPP} (see also
\cite{BDGO} for earlier existence results for equations with
general bounded measure $\mu$ and \cite{BM} for uniqueness results
in the case where $\mu$ is a function in $L^1$). In \cite{DP,Pe}
definitions of renormalized solutions to (\ref{eq1.2}) with
Leray-Lions type operators and $f$ depending on $u$ have been
proposed (in \cite{Pe} equations with general, not necessarily
smooth measures are considered). Another definition of a
renormalized solution, which is suitable for handling equations
with local operators and nonlinear $f$, have been introduced in
\cite{PPP}. It may be viewed as parabolic analogue of
(\ref{eq1.4}). Existence and uniqueness results for weak solutions
to linear equations with fractional Laplacian and $\mu$ being a
function in $L^1$ are proved in \cite{LPPS}. A probabilistic
approach to (\ref{eq1.2}) has been developed in \cite{K:JFA}. A
probabilistic solution of (\ref{eq1.2})  is defined similarly to
(\ref{eq1.3}), but with $\BM$ replaced by a time-space Markov
process associated with $\EE^{0,T}$. In \cite{K:JFA} the
existence, uniqueness and regularity of probabilistic solutions of
(\ref{eq1.2}) is proved for $f$ satisfying some natural conditions
(monotonicity together with mild integrability conditions) and
general operators associated with $\EE^{0,T}$.

Similarly to the elliptic case, in the paper we generalize the
notion of a renormalized solution of \cite{PPP} to the case of
general operators corresponding to $\EE^{0,T}$. Then we  show that
the proposed definition is equivalent to the probabilistic
definition considered in \cite{K:JFA}. As in elliptic case, this
shows that the  renormalized solutions are properly defined and
gives new information on the structure of solutions. We illustrate
the utility of our  result by stating some theorems on existence
and uniqueness of renormalized solutions of parabolic equations
with $f$ satisfying the monotonicity condition and mild
integrability conditions.

For simplicity, in the paper we confine ourselves to equations
with operators corresponding to  regular forms, but our results
can be generalized to quasi-regular forms (see remarks at the end
of Sections \ref{sec3} and \ref{sec4}).

\section{Preliminaries} \label{sec2}

In the paper we assume that $E$ is a locally compact separable
metric space and $m$ is an everywhere dense Radon measure on $E$,
i.e. $m$ is a non-negative Borel measure on $E$ finite on compact
sets and strictly positive on non-empty open sets.

We set $E^1=\BR\times E$, $E_T=[0,T]\times E$,
$E_{0,T}=(0,T]\times E$. By $\BB(E)$ we denote the $\sigma$-field
of Borel subsets of $E$. $\BB_b(E)$ is the set of all real bounded
Borel measurable functions on $E$ and $\BB_b^+(E)$ is the subset
of $\BB_b(E)$ consisting of positive functions. The sets
$\BB(E^1)$, $\BB_b(E^1)$, $\BB_b^+(E^1)$ are defined analogously.

We set $H=L^2(E;m)$ and $\HH_{0,T}=L^2(0,T;H)$. The last space we
identify with $L^2(E_T;m_1)$, where $m_1=dt\otimes m$. By
$(\cdot,\cdot)_H$,  $(\cdot,\cdot)_{\HH_{0,T}}$ we denote the
usual inner products in $H$ and $\HH_{0,T}$, respectively.

\subsection{Dirichlet forms}
\label{sec2.1}

In what follows we assume that $(\EE,D(\EE))$ is a (non-symmetric)
Dirichlet form on $H$, i.e. positive definite closed form
satisfying the weak sector condition and such that $(\EE,D(\EE))$
has both the sub-Markov and the dual sub-Markov property. For the
definitions we refer the reader to \cite{MR}. Here let us only
recall hat $(\EE,D(\EE))$ satisfies the weak sector condition if
there is $K>0$ (called the sector constant) such that
\[
|\EE_1(u,v)|\le K\EE_1(u,u)^{1/2}\EE_1(v,v)^{1/2},\quad u,v\in
D(\EE),
\]
where $\EE_{\alpha}(u,v)=\EE(u,v)+\alpha(u,v)_H$ for $\alpha\ge0$.
If there is $K>0$ such that
\begin{equation}
\label{eq2.9} |\EE(u,v)|\le K\EE(u,u)^{1/2}\EE(v,v)^{1/2},\quad
u,v\in D(\EE),
\end{equation}
then we say that $(\EE,D(\EE))$ satisfies the strong sector
condition.

By Theorems I.2.8 and I.4.4 in \cite{MR} every Dirichlet form on
$H$ determines uniquely strongly continuous contraction resolvents
$(G_{\alpha})_{\alpha>0}$, $(\hat G_{\alpha})_{\alpha>0}$ on $H$
such that $G_{\alpha}, \hat G_{\alpha}$ are sub-Markov,
$G_{\alpha}(H)\subset D(\EE)$, $\hat G_{\alpha}(H)\subset D(\EE)$
and
\[
\EE_{\alpha}(G_{\alpha}f,u)=(f,u)_{H}=\EE_{\alpha}(u,\hat
G_{\alpha}f),\quad f\in H,\, u\in D(\EE),\,\alpha>0.
\]
In fact, from the sub-Markov and the dual sub-Markov property of
$(\EE,D(\EE)$  it follows that $(G_{\alpha})_{\alpha>0},(\hat
G_{\alpha})_{\alpha>0}$ may be extended to sub-Markov resolvents
on $L^{\infty}(E;m)$ and on $L^{1}(E;m)$, respectively (see
\cite[Section 1.1]{O3}).

Let $f\in L^{\infty}(E;m)$ be a non-negative function. Since
$G_{1/l}f$ increases as $l\uparrow\infty$, the potential operator
\[
Gf=\lim_{l\rightarrow\infty}G_{1/l}f
\]
is $m$-a.e. well defined  but may take the value $\infty$. We say
that $\EE$ is transient if $Gf<\infty$ $m$-a.e. for every
non-negative $f\in L^{\infty}(E;m)$.

Let $\tilde \EE$ denote the symmetric part of $\EE$, i.e.
$\tilde\EE(u,v)=\frac12(\EE(u,v)+\EE(v,u))$. The extended
Dirichlet space $D_e(\EE)$ associated with a Dirichlet form
$(\EE,D(\EE))$ is the family of measurable functions
$u:E\rightarrow\BR$ such that $|u|<\infty$ $m$-a.e. and there
exists an $\tilde\EE$-Cauchy sequence $\{u_n\}\subset D(\EE)$ such
that $u_n\rightarrow u$ $m$-a.e. The sequence $\{u_n\}$ is called
an approximating sequence for $u\in D_e(\EE)$.

For $u\in D_e(\EE)$ we set
$\EE(u,u)=\lim_{n\rightarrow\infty}\EE(u_n,u_n)$, where $\{u_n\}$
is an approximating sequence for $u$. If moreover $\EE$ satisfies
the strong sector condition (\ref{eq2.9}) then we may extend $\EE$
to $D_e(\EE)$ by putting
$\EE(u,v)=\lim_{n\rightarrow\infty}\EE(u_n,v_n)$ with
approximating sequences $\{u_n\}$ and $\{v_n\}$ for $u\in
D_e(\EE)$ and $v\in D_e(\EE)$, respectively (see \cite[Section
1.3]{O3}). This extension satisfies again the strong sector
condition. By \cite[Theorem 1.3.9]{O3}, if $(\EE,D(\EE))$ is
transient then $(D_e(\EE),\tilde\EE)$ is a Hilbert space.

Given a Dirichlet form $(\EE,D(\EE))$ we define quasi notions with
respect to $\EE$ (exceptional sets, nests and quasi-continuity) as
in \cite[Chapter III]{MR} (see also \cite[Sections 2.1, 2.2]{O3}).
We will say that a property of points in $E$ holds quasi
everywhere (q.e. for short) if it holds outside some exceptional
set.

In the paper we assume that $(\EE,D(\EE))$ is regular (see
\cite[Section IV.4]{MR} or \cite[Section 1.2]{O3} for the
definition). By \cite[Proposition IV.3.3]{MR}, if $(\EE,D(\EE))$
is a regular Dirichlet form then each element $u\in D(\EE)$ admits
an quasi-continuous $m$-version, which we denote by $\tilde u$,
and $\tilde u$ is q.e. unique for every $u\in D(\EE)$. If moreover
$(\EE,D(\EE))$ is transient then such a unique $m$-version $\tilde
u$ exists for every $u\in D_e(\EE)$. This follows from
\cite[Theorem 2.1.7]{FOT} and the fact that $D_e(\EE)$ and the
notion of quasi-continuity only depend on the symmetric part of
$\EE$.

A positive measure $\mu$ on $\BB(E)$ is said to be smooth ($\mu\in
S(E)$ in notation) if $\mu(B)=0$ for all exceptional sets
$B\in\BB(E)$ and there exists an nest $\{F_k\}_{k\in\BN}$ of
compact sets such that $\mu(F_k)<\infty$ for $k\in\BN$.

Given a transient form $(\EE,D(\EE))$ satisfying the strong sector
condition we will denote by $S^{(0)}_0(E)$ the set of measures of
finite $0$-order energy integral, i.e. the subset of $S(E)$
consisting of all measures $\nu\in S(E)$ such that for some $c>0$,
\[
\int_E|\tilde u(x)|\,\nu(dx)\le c\EE(u,u)^{1/2},\quad u\in
D_e(\EE).
\]
If $\nu\in S^{(0)}_{0}(E)$ then from the Lax-Milgram theorem it
follows that there is a unique element $\hat U\nu$ (called a
copotential of $\nu$) such that
\[
\EE(u,\hat U\nu)=\int_E\tilde u(x)\,\nu(dx),\quad u\in D_e(\EE).
\]
By $\hat S^{(0)}_{00}(E)$ we  denote the subset of $S^{(0)}_0(E)$
consisting of all measures $\nu$ such that $\nu(E)<\infty$ and
$\|\hat U\nu\|_{\infty}<\infty$.

\subsection{Time dependent Dirichlet forms}
\label{sec2.2}

We assume that we are given a family $\{B^{(t)},t\in[0,T]\}$ of
Dirichlet forms on $H$ with common domain $V$ and sector constant
$K$  independent of $t$. We also assume that
\begin{enumerate}
\item[(a)] $[0,T]\ni t\mapsto B^{(t)}(\varphi,\psi)$ is measurable for every
$\varphi,\psi\in V$,
\item[(b)] there is a constant $\lambda\ge1$ such that
$\lambda^{-1}B(\varphi,\varphi)\le
B^{(t)}(\varphi,\varphi)\le\lambda B(\varphi,\varphi) $ for every
$t\in[0,T]$ and $\varphi\in V$, where
$B(\varphi,\varphi)=B^{(0)}(\varphi,\varphi)$.
\end{enumerate}
By putting $B^{(t)}=B$ for $t\in\BR\setminus[0,T]$ we may and will
assume that $B^{(t)}$ is defined and satisfies (a), (b) for
$t\in\BR$.

By the definition of a Dirichlet form $V$ is a dense subspace of
$H$ and $(B,V)$ is closed. Therefore $V$ is a real Hilbert space
with respect to $\tilde B_1(\cdot,\cdot)$, which is densely and
continuously embedded in $H$. By $\|\cdot\|_V$ we denote the norm
in $V$, i.e. $\|\varphi\|^2_V=B_1(\varphi,\varphi)$, $\varphi\in
V$. By $V'$ we denote the dual space of $V$ and by
$\|\cdot\|_{V'}$  the corresponding norm. We set $\HH=L^2(\BR;H)$,
$\VV=L^2(\BR;V)$, $\VV'=L^2(\BR;V')$ and
\begin{equation}
\label{eq2.3} \|u\|^2_{\VV}=\int_{\BR}\|u(t)\|^2_V\,dt,\quad
\|u\|^2_{\VV'}=\int_{\BR}\|u(t)\|^2_{V'}\,dt.
\end{equation}
We shall identify  $H$ and its dual $H'$. Then $V\subset H\simeq
H'\subset V'$ continuously and densely, and hence
$\VV\subset\HH\simeq\HH'\subset\VV'$ continuously and densely.
%, and by (B2)(?),
%\[
%\|u\|_{\VV'}\|u\|_{\HH}\le \|u\|_{\VV} \quad ? (potrzebne?)
%\]

For $u\in\VV$ we denote by $\frac{\partial u}{\partial t}$ the
derivative in the distribution  sense of the function $t\mapsto
u(t)\in V$ and we set
\begin{equation}
\label{eq2.4} \WW=\{u\in \VV:\frac{\partial u}{\partial t}\in
\VV'\}, \quad \|u\|_{\WW}=\|u\|_{\VV}+\|\frac{\partial u}{\partial
t}\|_{\VV'}
\end{equation}

We will consider time dependent Dirichlet forms $\EE$ and
$\EE^{0,T}$  associated with the families
$\{(B^{(t)},V),t\in\BR\}$ and $\{(B^{(t)},V),t\in[0,T]\}$,
respectively. We define $\EE$ by
\begin{equation}
\label{eq2.23} \mathcal{E}(u,v)=\left\{
\begin{array}{l}\langle-\frac{\partial u}{\partial t},v\rangle+\BB(u,v),
\quad u\in\WW,v\in\VV,\smallskip \\
\langle\frac{\partial v}{\partial t}, u\rangle+\BB(u,v),\quad
u\in\VV,v\in\WW,
\end{array}
\right.
\end{equation}
where $\langle\cdot,\cdot\rangle$ is the duality pairing between
$\VV'$ and $\VV$ and
\[
\BB(u,v)=\int_{\BR}B^{(t)}(u(t),v(t))\,dt.
\]
Note that $\EE$ can be identified with some generalized Dirichlet
form  (see \cite[Example I.4.9.(iii)]{S}).

Given a time dependent form (\ref{eq2.23}) we define capacity as
in \cite[Section 6.2]{O3}, and then using it we define
quasi-notions (exceptional sets, nests and quasi-continuity) as in
\cite[Section 6.2]{O3}. Note that by \cite[Theorem 6.2.11]{O3}
each element $u$ of $\WW$ has a quasi-continuous $m_1$-version. We
will denote it by $\tilde u$.

To define $\EE^{0,T}$, we set $\HH_{0,T}=L^2(0,T;H)$,
$\VV_{0,T}=L^2(0,T;V)$, $\VV'_{0,T}=L^2(0,T;V')$ and
$\WW_{0,T}=\{u\in \VV_{0,T}:\frac{\partial u}{\partial t}\in
\VV'_{0,T}\}$ (the norms in $\VV_{0,T}$, $\VV'_{0,T}$, $\WW_{0,T}$
are defined analogously to (\ref{eq2.3}), (\ref{eq2.4})). Let
$C([0,T];H)$ denote the space of all continuous functions on
$[0,T]$ with values in $H$ equipped with the norm
$\|u\|_C=\sup_{0\le t\le T}\|u(t)\|_H$. It is known (see, e.g.,
\cite[Theorem 2]{ZKO} that there is a continuous embedding of
$\WW_{0,T}$ into $C([0,T];H)$, i.e. for every $u\in\WW_{0,T}$ one
can find  $\bar u\in C([0,T];H)$ such that $u(t)=\bar u(t)$ for
a.e. $t\in[0,T]$ (with respect to the Lebesgue measure) and
\begin{equation}
\label{eq2.20} \|\bar u\|_C\le M\|u\|_{\WW_{0,T}}
\end{equation}
for some $M>0$. In what follows we adopt the convention that any
element of $\WW_{0,T}$ is already in $C([0,T];H)$. With this
convention we may define the spaces
\[
\WW_0=\{u\in\WW_{0,T}:u(0)=0\},\quad
\WW_T=\{u\in\WW_{0,T}:u(T)=0\}.
\]
By the definition of $\WW_{0,T}$, ${\partial/\partial
t}:\WW_{0,T}\rightarrow\VV'_{0,T}$ is bounded. Since $\WW_0$ is
dense in $\VV_{0,T}$, we can regard the  restriction of
${\partial/\partial t}$ to $\WW_0$ as an unbounded operator from
$\VV_{0,T}$ to $\VV'_{0,T}$ defined on $\WW_0$. Its adjoint is
defined on $\WW_T$ and is given by $-{\partial/\partial t}$ (see,
e.g., \cite{ZKO}). Finally, we set
\begin{equation}
\label{eq2.6} \mathcal{E}^{0,T}(u,v)=\left\{
\begin{array}{l}\langle-\frac{\partial u}{\partial t},v\rangle
+\int_0^T B^{(t)}(u(t),v(t))\,dt,
\quad u\in\WW_T,v\in\VV_{0,T},\smallskip \\
\langle\frac{\partial v}{\partial t}, u\rangle+\int_0^T
B^{(t)}(u(t),v(t))\,dt,\quad u\in\VV_{0,T},v\in\WW_0,
\end{array}
\right.
\end{equation}
where now $\langle\cdot,\cdot\rangle$ denote the duality pairing
between $\VV'_{0,T}$ and $\VV_{0,T}$. As in the case of $\EE$, the
form $\EE^{0,T}$ can be identified with some generalized Dirichlet
form  (see \cite[Example I.4.9.(iii)]{S}).

By Propositions I.3.4 and I.3.6 in \cite{S} the form $\EE^{0,T}$
determines uniquely strongly continuous resolvents
$(G^{0,T}_{\alpha})_{\alpha>0}$, $(\hat
G^{0,T}_{\alpha})_{\alpha>0}$  on $\HH_{0,T}$ such that
$G^{0,T}_{\alpha}$, $\hat G^{0,T}_{\alpha}$ are sub-Markov,
$G^{0,T}_{\alpha}(\HH_{0,T})\subset\WW_T$, $\hat
G^{0,T}_{\alpha}(\HH_{0,T})\subset\WW_0$ and
\[
\EE^{0,T}_{\alpha}(G^{0,T}_{\alpha}\eta,u)
=(u,\eta)_{\HH_{0,T}},\quad \EE^{0,T}_{\alpha}(u,\hat
G^{0,T}_{\alpha}\eta)=(u,\eta)_{\HH_{0,T}},\quad u\in\VV_{0,T}
,\eta\in\HH_{0,T},
\]
where
$\EE^{0,T}_{\alpha}(u,v)=\EE^{0,T}(u,v)+\alpha(u,v)_{\HH_{0,T}}$
for $\alpha\ge0$.

\subsection{Markov processes and additive functionals}
\label{sec2.3}

In what follows $\Delta$ is a one-point compactification of $E$.
If $E$ is already compact then we adjoin $\Delta$ to $E$ as an
isolated point.

In the case of Dirichlet forms (and elliptic equations) we adopt
the convention that every function $f$ on $E$ is extended to
$E\cup\{\Delta\}$  by setting $f(\Delta)=0$.

In the case of time dependent Dirichlet forms (and parabolic
equations) we adopt the convention that every function $\varphi$
on $E$ is extended to $E^1$ by setting $\varphi(t,x)=\varphi(x)$,
$(t,x)\in E^1$, and every function $f$ on $E^1$ (resp. $E_{0,T}$)
is extended to $E^1\cup\{\Delta\}$ by setting $f(\Delta)=0$ (resp.
$f(z)=0$ for $z\in E^1\cup\{\Delta\}\setminus E_{0,T})$.

\paragraph{Dirichlet forms}

Let $(\EE,D(\EE))$ be a regular Dirichlet form on $H$. Then there
exists a unique  Hunt process
$\BM=(\Omega,(\FF_t)_{t\ge0},(X_t)_{t\ge0},\zeta,(P_x)_{x\in
E\cup\{\Delta\}})$  with state space $E$, life time $\zeta$ and
cemetery state $\Delta$  properly associated with $(\EE,D(\EE))$
(see Theorems IV.3.5, IV.6.4 and V.2.13 in \cite{MR}). The last
statement means that for every $\alpha>0$ and $f\in\BB_b(E)\cap H$
the resolvent of $\BM$, that is the function
\[
R_{\alpha}f(x)=E_x\int^{\infty}_0e^{-\alpha t}f(X_t)\,dt,\quad
x\in E,\,\alpha>0
\]
is a quasi-continuous $m$-version of $G_{\alpha}f$ (see
\cite[Proposition IV.2.8]{MR}).

It is known (see, e.g., \cite[Theorem VI.2.4]{MR}) that there is a
one to one correspondence (called Revuz correspondence) between
smooth measures $\mu$ and positive continuous additive functionals
(positive CAFs) $A$ of $\BM$. It is given by the following
relation
\begin{equation}
\label{eq2.15}
\lim_{t\rightarrow0^+}\frac1{t}E_m\int^t_0f(X_s)\,dA_s=\int_Ef(x)\,\mu(dx),
\quad f\in\BB^+(E),
\end{equation}
where $E_m$ denotes the expectation with respect to the measure
$P_m(\cdot)=\int_EP_x(\cdot)\,m(dx)$. In what follows the positive
CAF of $\BM$ corresponding to $\mu\in S(E)$ will be denoted by
$A^{\mu}$.

For $\mu\in S(E)$ we set
\[
R\mu(x)=E_x\int^{\zeta}_0dA^{\mu}_t,\quad x\in E
\]
and
\[
\mathcal{R}(E)=\{\mu:|\mu|\in S(E), R|\mu|<\infty\mbox{
$m$-a.e.}\},
\]
where $|\mu|$ denotes the total variation of $\mu$. Note that by
\cite[Lemma 2.3]{KR:CM}, in the above definition of the class
$\RR(E)$ one can replace $m$-a.e. by q.e. By
$\mathcal{M}_{0,b}(E)$ we denote the space of all signed measures
$\mu$ on $E$ such that $|\mu|\in S(E)$ and $|\mu|(E)<\infty$. By
\cite[Proposition 3.2]{KR:CM}, if $(\EE,D(\EE))$ is transient then
$\MM_{0,b}(E)\subset\RR(E)$.

\paragraph{Time dependent Dirichlet forms.}

Let us consider the time dependent Dirichlet form $\EE$ defined by
(\ref{eq2.23}).  Then by \cite[Theorem 6.3.1]{O3} there exists a
Hunt process
$\BBM=(\Omega,(\FF_t)_{t\ge0},(\BBX_t)_{t\ge0},\zeta,(P_z)_{z\in
E^1\cup\{\Delta\}})$  with state space $E^1$, life time $\zeta$
and cemetery state $\Delta$ properly associated with $\EE$ in the
sense that for every $\alpha>0$ and $f\in\BB_b\cap L^2(E^1;m_1)$
the resolvent of $\BBM$ defined as
\[
R_{\alpha}f(z)=E_z\int^{\infty}_0e^{-\alpha t}f(\BBX_t)\,dt,\quad
x\in E,\,\alpha>0
\]
is a quasi-continuous version of the resolvent $G_{\alpha}f$
associated with $\EE$. Moreover, by \cite[Theorem 6.3.1]{O3},
\[
\BBX_{t}=(\tau(t), X_{\tau(t)}), \quad t\ge 0,
\]
where $\tau(t)$ is the uniform motion to the right, i.e.
$\tau(t)=\tau(0)+t$, $\tau(0)=s$, $P_{z}$-a.s. for $z=(s,x)$.

Let $S(E^1)$ denote the set of smooth measures on $E^1$ (with
respect to $\EE$),  which we define analogously to $S(E)$ (see,
e.g., \cite{K:JFA,S} for details).  We say that a positive AF $A$
of $\BBM$ is in the Revuz correspondence with $\mu\in S(E^1)$ if
\[
\lim_{\alpha\rightarrow\infty}\alpha
E_{m_1}\int^{\infty}_0e^{-\alpha t}
f(\BBX_t)\,dA_t=\int_{E^1}f(z)\,\mu(dz),\quad f\in\BB_b^+(E^1),
\]
where $E_{m_1}$ denotes the expectation with respect to
$P_{m_1}(\cdot)=\int_{E^1}P_z(\cdot)\,m_1(dz)$ (see \cite{O1,R}).

It is known (see \cite[Section 2]{K:JFA}) that for every $\mu\in
S(E^1)$ there exists a unique positive  natural AF $A$ of $\BBM$,
i.e. a positive AF of $\BBM$ such that $A$ and $\BBM$ have no
common discontinuities, such that $A$ is in the Revuz
correspondence with $\mu$. In what follows we will denote it by
$A^{\mu}$. In fact, $A^{\mu}$ is a predictable process (see
\cite{GS}). On the contrary, if $A$ is a positive natural AF of
$\BBM$ then by Proposition in Section II.1 of \cite{R} and
\cite[Theorem 5.6]{O1} there exists a smooth measure $\mu$ on
$E^1$ such that $A$ is in the Revuz correspondence with $\mu$.

Let $S(E_{0,T})$  denote the set of all $\mu\in S(E^1)$ with
support in $E_{0,T}$ and for $\mu\in S(E_{0,T})$ let
\[
R^{0,T}\mu(z)=E_z\int_0^{\zeta_{\tau}}dA^{\mu}_t,\quad z\in
E_{0,T},
\]
where
\begin{equation}
\label{eq2.10} \zeta_{\tau}=\zeta\wedge(T-\tau(0)).
\end{equation}
We set
\[
\mathcal{R}(E_{0,T})=\{\mu:|\mu|\in S(E_{0,T}),
R^{0,T}|\mu|<\infty \mbox{ $m_1$-a.e.}\}
\]
and by $\mathcal{M}_{0,b}(E_{0,T})$ we denote the space of all
signed measures $\mu$ on $E^1$ such that $|\mu|\in S(E_{0,T})$ and
$|\mu|(E^1)<\infty$. Note that by \cite[Proposition 3.8]{K:JFA},
$\MM_{0,b}(E_{0,T})\subset\RR(E_{0,T})$.

\section{Elliptic equations} \label{sec3}

Let $(\EE,D(\EE))$ be a regular Dirichlet form on $H$. We consider
the problem
\begin{equation}
\label{eq3.01} -Lu=f_u+\mu,
\end{equation}
where $f:E\times\BR\rightarrow\BR$ is a measurable function,
$f_u=f(\cdot,u)$, $\mu\in\RR(E)$ and $L$ is the operator
associated with $(\EE,D(\EE))$, i.e.
\[
D(L)=\{u\in D(\EE):v\mapsto\EE(u,v)\mbox{ is continuous with
respect to $(\cdot,\cdot)^{1/2}_H$ on $D(\EE)$}\}
\]
and
\begin{equation}
\label{eq1.6} (-Lu,v)_H=\EE(u,v),\quad u\in D(L),v\in D(\EE)
\end{equation}
(see \cite[Proposition I.2.16]{MR}).

In what follows  $\BM$ is the Markov process of Section \ref{sec2}
associated with $(\EE,D(\EE))$. Let us recall that a c\`adl\`ag
adapted (with respect to $(\FF_t)$) process $Y$ is said to be of
class (D)  if the collection $\{Y_{\tau},\tau$ is a finite
$(\FF_t)$-stopping time\} is uniformly integrable.

\begin{definition}
Let $f:E\times\BR\rightarrow\BR$ be a measurable function and let
$A^{\mu}$ be a CAF of $\BM$ corresponding to some
$\mu\in\mathcal{R}(E)$. We say that a pair $(Y^x,M^x)$ is a
solution of the backward stochastic differential equation
\begin{equation}
\label{eq3.2}
Y^x_t=Y^x_{T\wedge\zeta}+\int^{T\wedge\zeta}_{t\wedge\zeta}f(X_s,Y^x_s)\,ds
+\int^{T\wedge\zeta}_{t\wedge\zeta}dA^{\mu}_s
-\int^{T\wedge\zeta}_{t\wedge\zeta}dM^x_s, \quad t\ge0
\end{equation}
under the measure $P_x$ if
\begin{enumerate}
\item[(a)]$Y^x$ is an $(\FF_t)$-progressively measurable c\`adl\`ag
process such that $Y^x_{t\wedge\zeta}\rightarrow0$, $P_x$-a.s. as
$t\rightarrow\infty$, $Y^x$ is of class (D) under $P_x$ and $M^x$
is a c\`adl\`ag $(\FF_t)$-local martingale under $P_x$,
\item[(b)]For every $T>0$, $[0,T]\ni t\mapsto f(X_t,Y^x_t)\in L^1(0,T)$ and
(\ref{eq3.2}) is  satisfied $P_x$-a.s.
\end{enumerate}
\end{definition}

The following definition is taken from \cite{KR:JFA,KR:CM}.
\begin{definition}
Let $\mu\in\RR(E)$. We say that a quasi-continuous function
$u:E\rightarrow\BR$ is a probabilistic solution to (\ref{eq1.1})
if $f_u\cdot m\in\RR(E)$ and for q.e. $x\in E$,
\begin{equation}
\label{eq3.1} u(x)=E_x\Big(\int^{\zeta}_0f_u(X_t)\,dt
+\int^{\zeta}_0dA^{\mu}_t\Big).
\end{equation}
\end{definition}

\begin{remark}
\label{rem3.1} (i) The quasi-continuity requirement on $u$ in the
above definition can be omitted, because  if $\mu,f_u\cdot
m\in\RR(E)$ then from the  very definition of the class $\RR(E)$
it follows that  the right-hand side of (\ref{eq3.1}) is finite
for $m$-a.e. $x\in E$, and, in consequence, it is a
quasi-continuous function of $x$ (see \cite[Lemma 4.3]{KR:JFA} and
\cite[Lemma 2.3]{KR:CM}).
\smallskip\\
(ii) If $u$ is a probabilistic solution to (\ref{eq1.1}) then
there exists a martingale additive functional (MAF) $M$ of $\BM$
such that $M$ is a martingale under $P_x$ for q.e. $x\in E$ and
for q.e. $x\in E$ the pair
\[
(Y_t,M_t)=(u(X_t),M_t),\quad t\ge0
\]
is a solution of (\ref{eq3.2}) under $P_x$. Indeed, with our
convention (see the beginning of Section \ref{sec2.3}),
\[
u(x)=E_x\Big(\int^{\infty}_0f_u(X_t)\,dt
+\int^{\infty}_0dA^{\mu}_t\Big).
\]
Set
\[
M^x_t=E_x\Big(\int^{\zeta}_0f_u(X_s)\,ds
+\int^{\zeta}_0dA^{\mu}_s|\FF_{t\wedge\zeta}\Big)-u(X_0),\quad
t\ge0.
\]
By \cite[Lemma A.3.6]{FOT} there exists a  MAF $M$ od $\BM$ such
that $M^x_t=M_t$, $t\ge0$, $P_x$-a.s. for q.e. $x\in E$. Therefore
\begin{equation}
\label{eq3.04}
M_t=M_{t\wedge\zeta}=E_x\Big(\int^{\zeta}_0f_u(X_s)\,ds
+\int^{\zeta}_0dA^{\mu}_s|\FF_{t\wedge\zeta}\Big)-u(X_0),\quad
t\ge0
\end{equation}
under $P_x$ for q.e. $x\in E$. By the strong Markov property,
under $P_x$  we have
\begin{align*}
M_{t\wedge\zeta}&=\int^{t\wedge\zeta}_0(f_u(X_s)\,ds +dA^{\mu}_s)
+E_x\Big(\int^{\zeta}_{t\wedge\zeta}(f_u(X_s)\,ds
+dA^{\mu}_s)|\FF_{t\wedge\zeta}\Big)-u(X_0)\\
&=\int^{t\wedge\zeta}_0(f_u(X_s)\,ds +dA^{\mu}_s)
+E_{X_{t\wedge\zeta}}\int^{\zeta}_0(f_u(X_s)\,ds
+dA^{\mu}_s)-u(X_0)\\
&=\int^{t\wedge\zeta}_0(f_u(X_s)\,ds
+dA^{\mu}_s)+u(X_{t\wedge\zeta})-u(X_0)
\end{align*}
for q.e. $x\in E$. Hence
\[
u(X_{t\wedge\zeta})-u(X_{T\wedge\zeta})
=\int^{T\wedge\zeta}_{t\wedge\zeta}(f_u(X_s)\,ds +dA^{\mu}_s)
-\int^{T\wedge\zeta}_{t\wedge\zeta}dM_s,
\]
which shows (\ref{eq3.2}). Taking  $t=0$ in the above equality and
using (\ref{eq3.04}) we get
\[
u(X_{T\wedge\zeta})=-\int^{T\wedge\zeta}_0(f_u(X_s)\,ds
+dA^{\mu}_s)+E_x\Big(\int^{\zeta}_0f_u(X_s)\,ds
+\int^{\zeta}_0dA^{\mu}_s|\FF_{T\wedge\zeta}\Big).
\]
It follows that for q.e. $x\in E$,
$u(X_{T\wedge\zeta})\rightarrow0$, $P_{x}$-a.s. as
$T\rightarrow\infty$.
\end{remark}

In what follows we assume that $(\EE,D(\EE))$ is transient and
satisfies the strong sector condition. For a measure $\mu$ on $E$
and a function $u:E\rightarrow\BR$ we use the notation
\[
\langle\mu,u\rangle=\int_Eu(x)\,\mu(dx)
\]
whenever the integral is well defined.

We adopt the following definition of renormalized solution of
(\ref{eq1.1}). In the case of local operators, this is essentially
\cite[Definition 2.29]{DMOP}.

\begin{definition}
\label{def3.8} Let $\mu\in\MM_{0,b}(E)$. We say that
$u:E\rightarrow\BR$ is a renormalized solution of (\ref{eq1.1}) if
\begin{enumerate}
\item[(a)] $u$ is quasi-continuous, $f_u\in L^1(E;m)$
and $T_ku\in D_e(\EE)$ for every $k>0$,
\item[(b)] there exists a sequence
$\{\nu_k\}\subset\MM_{0,b}(E)$ such that
$\|\nu_{k}\|_{TV}\rightarrow0$ as $k\rightarrow\infty$ and for
every $k\in\BN$ and every bounded $v\in D_e(\EE)$,
\begin{equation}
\label{eq4.1} \EE(T_ku,v) =\langle f_u\cdot m+\mu,\tilde v\rangle+
\langle\nu_{k},\tilde v\rangle.
\end{equation}
\end{enumerate}
\end{definition}

\begin{theorem}
\label{th3.5} Assume that $(\EE,D(\EE))$ is transient, satisfies
the strong sector condition and that $\mu\in\MM_{0,b}(E)$.
\begin{enumerate}
\item[\rm(i)]If $u$ is a probabilistic solution of \mbox{\rm(\ref{eq1.1})}
and $f_u\in L^1(E;m)$ then $u$ is a renormalized solution of
\mbox{\rm(\ref{eq1.1})}.
\item[\rm(ii)]If $u$ is a renormalized solution of
\mbox{\rm(\ref{eq1.1})} then $u$ is a probabilistic solution of
\mbox{\rm(\ref{eq1.1})}.
\end{enumerate}
\end{theorem}
\begin{proof}
(i)  Let $u$ be a probabilistic solution of (\ref{eq1.1}) and let
$M$ be the martingale of Remark \ref{rem3.1}(ii). For $k>0$ put
\[
Y_t=u(X_t),\quad Y^k_t=T_ku(X_t),\quad t\ge0.
\]
From the fact that $(Y,M)$ is a solution of (\ref{eq3.2}) it
follows that
\begin{equation}
\label{eq4.2}
Y_t=Y_{t\wedge\zeta}=Y_0-\int^{t\wedge\zeta}_0f(X_s,Y_s)\,ds
-\int^{t\wedge\zeta}_0dA^{\mu}_s +\int^{t\wedge\zeta}_0dM_s, \quad
t\ge0.
\end{equation}
By the Meyer-It\^o formula (see, e.g., \cite[Theorem IV.70]{Pr}),
%\\
%K: \\
%Funkcja $f(y)=y\wedge k$ jest wkl/es/la, $f'(y)=1$ dla $y\le k$,
%$f'(y)=0$ dla $y>k$; $f''(dy)=-\delta_k(dy)$, bo $\langle
%f'',\varphi\rangle=-\langle
%f',\varphi'\rangle=\int^k_{-\infty}\varphi'(x)\,dx
%=-\varphi(y)|^{k}_{-\infty}
%=-\varphi(k)=-\langle\delta_k,\varphi\rangle$. Podobnie, ..
%\\
%KK
\begin{equation}
\label{eq4.02} (u\wedge k)(X_t)-(u\wedge k)(X_0)
=\int^t_0\fch_{\{Y_{s-}\le k\}}\,dY_s-A^{1,k}_t
\end{equation}
and
\begin{equation}
\label{eq4.05} (u(X_t)+k)\wedge0-(u(X_0)+k)\wedge0=
\int^t_0\fch_{\{Y_{s-}\le -k\}}\,dY_s-A^{2,k}_t
\end{equation}
for some increasing processes $A^{1,k},A^{2,k}$. Since
$T_ky=y\wedge k-((y+k)\wedge0)$ for $y\in\BR$, it follows from
(\ref{eq4.02}) and  (\ref{eq4.05})  that
\begin{equation}
\label{eq4.06} Y^k_t-Y^k_0=\int^t_0\fch_{\{-k<Y_{s-}\le k\}}\,dY_s
-(A^{1,k}_t-A^{2,k}_t).
\end{equation}
From (\ref{eq4.02}), it follows immediately that $A^{1,k},A^{2,k}$
are AFs of $\BM$. Since $u$ is a probabilistic solution,
$u(X_t)\rightarrow0$, $P_x$-a.s. as $t\rightarrow\infty$ for q.e.
$x\in E$. Therefore from (\ref{eq4.02}) and continuity of
$A^{\mu}$ we conclude that for q.e. $x\in E$,
\[
E_xA^{1,k}_{\zeta}=E_x(u\wedge k)(X_0)
-E_x\int^{\zeta}_0\fch_{\{Y_{s}\le k\}}(f_u(X_s)\,ds +dA^{\mu}_s).
\]
Since $E_x(u\wedge k)(X_0)=(u\wedge k)(x)\le u(x)$ and $u$ is a
probabilistic solution of (\ref{eq1.1}), it follows from the above
that
\begin{equation}
\label{eq3.15}  E_xA^{1,k}_{\zeta}\le
E_x\int^{\zeta}_0\fch_{\{Y_{s}>k\}}(f_u(X_s)\,ds +dA^{\mu}_s).
\end{equation}
Similarly, by  (\ref{eq4.05}) we have
\begin{equation}
\label{eq3.22} E_xA^{2,k}_{\zeta}\le
-E_x\int^{\zeta}_0\fch_{\{Y_s\le-k\}}(f_u(X_s)\,ds +dA^{\mu}_s).
\end{equation}
It follows that for q.e. $x\in E$,
$E_x(A^{1,k}_\zeta+A^{2,k}_\zeta)<\infty$. Therefore by
\cite[Theorem A.3.16]{FOT} there exists AFs $B^{1,k},B^{2,k}$ of
$\BM$ such that $B^{i,k}$ is a compensator of $A^{i,k}$, $i=1,2$
under $P_x$ for q.e. $x\in E$. Since $A^{1,k}, A^{2,k}$ are
increasing, $B^{1,k}, B^{2,k}$ are increasing, too. Furthermore,
since by \cite[Theorem A.3.2]{FOT} the process $X$ has no
predictable jumps, it follows from \cite[Theorem A.3.5]{FOT} that
$B^{1,k}, B^{2,k}$ are continuous. Thus $B^{1,k}, B^{2,k}$ are
increasing CAFs of $\BM$ such that $A^{i,k}-B^{i,k}$, $i=1,2$, are
martingales under $P_x$ for q.e. $x\in E$. Let $b^i_k$, $i=1,2$,
denote the Revuz measure of $B^{i,k}$. Since for q.e. $x\in E$,
$E_xA^{i,k}_t=E_xB^{i,k}_t$ for $t\ge0$, from (\ref{eq3.15}),
(\ref{eq3.22})  and \cite[Lemma 5.4]{KR:JFA} we conclude that
\[
b^1_k(E)\le
\|\fch_{\{u>k\}}f_u\|_{L^1(E;m)}+\|\fch_{\{u>k\}}\cdot\mu\|_{TV},
\]
\[
b^2_k(E)\le\|\fch_{\{u\le-k\}}f_u\|_{L^1(E;m)}
+\|\fch_{\{u\le-k\}}\cdot\mu\|_{TV}.
\]
Hence $b^1_k,b^2_k\in\MM_{0,b}(E)$ and
$\|b^1_k\|_{TV}\rightarrow0$, $\|b^2_k\|_{TV}\rightarrow0$ as
$k\rightarrow\infty$. Combining (\ref{eq4.2}) with (\ref{eq4.06})
we obtain
\begin{equation}
\label{eq4.19} Y^{k}_t-Y^{k}_0=-\int^t_0\fch_{\{-k<Y_{s-}\le
k\}}(f_u(X_s)\,ds+dA^{\mu}_s) -(B^{1,k}_t-B^{2,k}_t)+M^{k}_t
\end{equation}
with
\[
M^{k}_t=\int^t_0\fch_{\{-k<Y_{s-}\le k\}}
\,dM_s-(A^{1,k}_t-B^{1,k}_t) +A^{2,k}_t-B^{2,k}_t.
\]
From (\ref{eq4.19}) and the fact that $u(X_t)\rightarrow0$,
$P_x$-a.s. as $t\rightarrow\infty$ for q.e. $x\in E$ it follows
that for q.e. $x\in E$,
\begin{equation}
\label{eq3.20} T_{k}u(x)=E_x\Big(\int^{\zeta}_0(f_u(X_t)\,dt
+dA^{\mu}_t)+\int^{\zeta}_0dA^{\nu_k}_t\Big)
\end{equation}
with
\[
\nu_k=-\fch_{\{u\notin(-k,k]\}}(f_u\cdot m+\mu)+b^1_k-b^2_k.
\]
Clearly $\nu_k\in\MM_{0,b}(E)$ and $\|\nu_k\|_{TV}\rightarrow0$ as
$k\rightarrow\infty$. By \cite[Theorem 4.2]{KR:CM} (see also
\cite[Proposition 5.9]{KR:JFA} in the case of regular symmetric
forms), $T_{k}u\in D_e(\EE)$ for every $k>0$. Let
$\lambda_k=f_u\cdot m+\mu+\nu_k$ and let $A=A^{\lambda_k}$. Since
$\lambda_k\in\MM_{0,b}(E)$, $R|\lambda_k|(x)<\infty$ for q.e.
$x\in E$. By Fubini's theorem, for q.e. $x\in E$ we have
\begin{align*}
R{\lambda_k}(x)-R_{\alpha}\lambda_k(x)
=E_x\int^{\infty}_0(1-e^{-\alpha t})\,dA_t
&=E_x\int^{\infty}_0(\int^t_0\alpha e^{-\alpha s}\,ds)\,dA_t\\
&=\alpha E_x\int^{\infty}_0e^{-\alpha t}(\int_t^{\infty}dA_s)\,dt.
\end{align*}
By the Markov property the right-hand side of the above equality
equals
\begin{align*}
\alpha E_x\int^{\infty}_0e^{-\alpha t}(\int_0^{\infty}
d(A_s\circ\theta_t))\,dt &=\alpha
E_x\int^{\infty}_0e^{-\alpha t}E_{X_t}(\int^{\infty}_0dA_s)\,dt\\
&=\alpha E_x\int^{\infty}_0e^{-\alpha t}R\lambda_k(X_t)\,dt
=\alpha R_{\alpha}(R\lambda_k)(x).
\end{align*}
Hence
\[
R{\lambda_k}(x)-R_{\alpha}\lambda_k(x) =\alpha
R_{\alpha}(R\lambda_k)(x)
\]
for q.e. $x\in E$.
%K:\\
%Powyzszy dowod podobny do dowodu \cite[Lemma 4.1.1]{O3}.\\
%KK\\
By (\ref{eq3.20}), $T_ku=R\lambda_k$. Therefore from the above
generalized resolvent equation it follows that for every bounded
$v\in D(\EE)$ we have
\begin{equation}
\label{eq3.21} \alpha(T_ku-\alpha R_{\alpha}(T_ku),v)_H =\alpha
(R_{\alpha}\lambda_k,v)_H.
\end{equation}
By \cite[Theorem 2.13(iii)]{MR} the left-hand side of
(\ref{eq3.21}) converges to $\EE(T_ku,v)$ as
$\alpha\rightarrow\infty$. The right-hand side is equal to
$\langle\lambda_k,\alpha\hat R_{\alpha}v\rangle$, where $\hat
R_{\alpha}$ denotes the resolvent of a Hunt process associated
with the form $(\hat\EE,D(\EE))$ defined as
$\hat\EE(u,v)=\EE(v,u)$, $u,v\in D(\EE)$. Since the functions
$\alpha\hat R_{\alpha}v$ are bounded uniformly in $\alpha>0$ and
by Propositions I.2.13(ii) and III.3.5 in \cite{MR} we may assume
that $\alpha\hat R_{\alpha}v\rightarrow\tilde v$ q.e. as
$\alpha\rightarrow\infty$, the right-hand side of (\ref{eq3.21})
converges to $\langle\lambda_k,\tilde v\rangle$ as
$\alpha\rightarrow\infty$. Thus (\ref{eq4.1}) is satisfied for
bounded $v\in D(\EE)$. Now assume that $v$ is bounded, say by $l$,
and $v\in D_e(\EE)$. Let $v_n$ be an approximating sequence for
$v$. Then $T_l(v_n)\rightarrow v$ $m$-a.e. and, by \cite[Corollary
1.6.3]{FOT}, in $(D_e(\EE),\tilde\EE)$ as $n\rightarrow\infty$.
Taking a subsequence if necessary we may assume that
$T_lv_n\rightarrow\tilde v$ q.e. By what has already been proved,
for $n\in\BN$ we have
\[
\EE(T_ku,T_lv_n)=\langle \lambda_k,T_l\tilde v_n\rangle.
\]
Letting $n\rightarrow\infty$ we get (\ref{eq4.1}), which completes
the proof of (i).
\smallskip\\
(ii) If  $u$ is a renormalized solution of (\ref{eq1.1}) then $u$
is quasi-continuous and  (\ref{eq4.1}) is satisfied for all
functions $v$ of the form $v=\hat U\nu$ with $\nu\in\hat
S^{(0)}_{00}(E)$. Hence
\[
\langle\nu,T_ku\rangle=\EE(T_ku,\hat U\nu)=\langle f_u\cdot m+\mu+
\nu_k,\widetilde{\hat U\nu}\rangle.
\]
Therefore $T_{k}u$ is a solution in the sense of duality (see
\cite[Section 3.3]{KR:CM} or \cite[Section 5]{KR:JFA} for the
definition) of the linear problem
\begin{equation}
\label{eq4.7} -L(T_{k}u)=f_u +\mu+\nu_{k}.
\end{equation}
By \cite[Proposition 3.9]{KR:CM} (or \cite[Proposition
5.3]{KR:JFA} in the case of symmetric forms) $T_{k}u$ is a
probabilistic solution of (\ref{eq4.7}). In particular
(\ref{eq3.20}) (with the measure $\nu_{k}$ of (\ref{eq4.7})) is
satisfied. Since $\|\nu_k\|_{TV}\rightarrow0$ as
$k\rightarrow\infty$, there is a subsequence (still denoted by
$k$) such that
\begin{equation}
\label{eq3.26}
R\nu_{k}(x)=E_x\int^{\zeta}_0dA^{\nu_{k}}_t\rightarrow0
\end{equation}
for $m$-a.e. $x\in E$. To see this, let us first observe that if
$\mu\in S^{(0)}_0(E)$ and $\tilde u\le c$ $\mu$-a.e., where
$u=\hat U\mu$, then $u\le c$ $m$-a.e. Indeed, we have
\[
\EE(u\wedge c,u)=\EE(u\wedge c,\hat U\mu)=\int_E(\tilde u\wedge
c)\,\mu(dx) =\int_E\tilde u\,\mu(dx)=\EE(u,u).
\]
Hence
\begin{equation}
\label{eq3.18} \EE(u-u\wedge c,u-u\wedge c)=\EE(u-u\wedge
c,u)-\EE(u-u\wedge c,u\wedge c)\le0,
\end{equation}
the last inequality being a consequence of \cite[Theorem
I.4.4]{MR} and the fact that $\EE$ is a Dirichlet form. By
(\ref{eq3.18}) and \cite[Theorem 1.6.2]{FOT}, $u-u\wedge c=0$
$m$-a.e., which shows that $u\le c$ $m$-a.e. Since $m\in S(E)$, by
the 0-order version of \cite[Theorem 2.2.4]{FOT} (see the comment
following \cite[Corollary 2.2.2]{FOT}) there exists a generalized
nest $\{F_n\}$ such that $\mu_n:=\fch_{F_n}\cdot m\in
S^{(0)}_0(E)$ for $n\in\BN$ and
$m(E\setminus\bigcup^{\infty}_{n=1}F_n)=0$. Let $F_{n,N}=\{x\in
F_n:\widetilde{\hat U\mu_n}(x)\le N\}$ and
$\mu_{n,N}=\fch_{F_{n,N}}\cdot\mu_n=\fch_{F_{n,N}}\cdot m$. Then
$\mu_{n,N}\in S^{(0)}_0(E)$ and $\widetilde{\hat U\mu_{n,N}} \le
\widetilde{\hat U\mu_n}\le N$ $\mu_n$-a.e. Therefore by the
observation made above, $\hat U\mu_{n,N}\le N$ $m$-a.e., and hence
$\widetilde{\hat U\mu_{n,N}}\le N$ q.e.  Moreover,
\[
\int_{F_{n,N}}R|\nu_k|(x)\,m(dx)=\EE(R|\nu_k|,\hat
U\mu_{n,N})=\langle|\nu_k|,\widetilde{\hat U\mu_{n,N}}\rangle \le
\|\nu_k\|_{TV}\|\hat U\mu_{n,N}\|_{\infty}.
\]
Hence for every $n,N\in\BN$,
\begin{equation}
\label{eq3.19}
\lim_{k\rightarrow\infty}\int_{F_{n,N}}R|\nu_k|(x)\,m(dx)=0.
\end{equation}
Since $\widetilde{\hat U\mu_n}$ is q.e. finite,
$m(F_n\setminus\bigcup_{N=1}^{\infty}F_{n,N})=m(\{x\in
F_n:\widetilde{\hat U\mu_n}(x)=\infty\}=0$. Therefore from
(\ref{eq3.19}) one can deduce that (\ref{eq3.26}) holds for
$m$-a.e. $x\in F_n$ for each $n\in\BN$. Since $m(E\setminus
\bigcup^{\infty}_{n=1}F_n)=0$, we see that (\ref{eq3.26}) holds
for $m$-a.e. $x\in E$. Letting $k\rightarrow\infty$ in
(\ref{eq3.20}) and using (\ref{eq3.26}) we conclude that
(\ref{eq3.1}) holds true for $m$-a.e. $x\in E$. In fact, since $u$
and the right-hand side of (\ref{eq3.1}) are quasi-continuous,
(\ref{eq3.1}) holds for q.e. $x\in E$, which completes the proof.
\end{proof}
\medskip

%From Theorem \ref{th3.5} and results on  probabilistic solutions
%of (\ref{eq1.1}) we immediately get results on renormalized
%solutions of (\ref{eq1.1}). By way of illustration, let

To illustrate the utility of Theorem \ref{th3.5} we now give some
results on existence and uniqueness of renormalized solutions of
(\ref{eq1.1}) with $f$ satisfying the monotonicity condition and
mild integrability conditions. To state the results we will need
the following hypotheses.

\begin{enumerate}
\item[(E1)]$f:E\times\BR\rightarrow\BR$ is measurable  and
$y\mapsto f(x,y)$ is continuous for every $x\in E$,
\item [(E2)]$(f(x,y_{1})-f(x,y_{2}))(y_{1}-y_{2})\le 0$ for every
$y_{1}, y_{2}\in \BR$ and $x\in E$,
\item [(E3)]$\mu\in\MM_{0,b}(E)$ and
$f(\cdot,y)\in L^1(E;m)$ for every $y\in\BR$.
\end{enumerate}

In what follows we assume that $(\EE,D(\EE))$ satisfies the
assumptions of Theorem \ref{th3.5}.
\begin{theorem}
Let $u_{1}, u_{2}$ be renormalized solutions of
\mbox{\rm(\ref{eq1.1})}  with the data $(f^{1},\mu_{1})$ and
$(f^{2},\mu_{2})$, respectively. Assume that $\mu_{1}\le\mu_{2}$
and either that $f^{1}(x,u_{1}(x))\le f^{2}(x,u_{1}(x))$ $m$-a.e.
and $f^{2}$ satisfies \mbox{\rm(E2)} or $f^{1}(x,u_{2}(x))\le
f^{2}(x,u_{2}(x))$ $m$-a.e. and $f^{1}$ satisfies \mbox{\rm(E2)}.
Then $u_{1}(x)\le u_{2}(x)$ for q.e. $x\in E$.
\end{theorem}
\begin{proof}
Follows from Theorem \ref{th3.5} and \cite[Proposition
4.9]{KR:JFA}.
\end{proof}

\begin{corollary}
\label{cor3.7} If  \mbox{\rm{(E2)}} is satisfied then there exists
at most one renormalized solution of \mbox{\rm{(\ref{eq1.1})}}.
\end{corollary}

\begin{theorem}
\label{th3.8} Assume \mbox{\rm{ (E1)--(E3)}}. Then there exists
renormalized solution of \mbox{\rm{(\ref{eq1.1})}}.
\end{theorem}
\begin{proof}
Follows from Theorem \ref{th3.5} and \cite[Theorem 3.8,
Proposition 3.10]{KR:CM} (see also \cite[Theorem 5.14]{KR:JFA}).
\end{proof}

We close this section with some general remarks on possible
generalization of our results and on their applicability.

An inspection of the proof of Theorem \ref{th3.5} reveals that it
only makes use  of some general results from the theory of
stochastic processes that are valid for general semimartingals,
some results from \cite{KR:CM}, which are proved for quasi-regular
forms and the fact that $\BM$ associated with $\EE$ in the
resolvent sense is a standard process (the fact that $\BM$ a Hunt
process is not needed). Therefore the proof of Theorem \ref{th3.5}
carries over to quasi-regular Dirichlet forms.

By using probabilistic methods one can prove that for many
interesting equations there exists a unique probabilistic solution
$u$ such that $f_u\in L^1(E;m)$. This can be done for instance for
$f$ satisfying (E1)--(E3).  For specific examples of local and
nonlocal operators $A$ satisfying our assumptions see, e.g.,
\cite{FOT,KR:JFA,KR:CM,MR}.  Then as in Corollary \ref{cor3.7} and
Theorems \ref{th3.8} above, a direct consequence of Theorem
\ref{th3.5} is that $u$ is a renormalized solution to
(\ref{eq1.1}), i.e. has clear analytical meaning, and that $u$ is
the unique renormalized solution. On the other hand, renormalized
solutions to (\ref{eq1.1}), which are obtained by analytical
methods, automatically have stochastic representation of the form
(\ref{eq3.1}). We may then use (\ref{eq3.1}) (and the theory of
BSDEs; see Remark \ref{rem3.1}) to study further properties of the
solution by probabilistic methods. For instance, probabilistic
methods are quite effective in proving comparison results and
hence  uniqueness.

It would be desirable to define renormalized solutions for
equations with general bounded measures $\mu$, at least for some
classes of nonlocal operators. Another interesting open problem is
to give other equivalent to Definition \ref{def3.8} analytical
definitions of a solution (like in the local case considered in
\cite{DMOP}).

\section{Parabolic equations}
\label{sec4}

In this section we assume that the family $\{B^{(t)},t\in[0,T]\}$
satisfies the assumptions of Section \ref{sec2.2} and $\EE^{0,T}$
is the time dependent Dirichlet form defined by (\ref{eq2.6}).  By
$L_t$ we denote the operator associated with $B^{(t)}$ in the
sense of (\ref{eq1.6}) and by $\frac{\partial u}{\partial t}+L_t$
the operator corresponding to $\EE^{0,T}$, i.e. the generator of
the strongly continuous contraction semigroup corresponding to
$(G^{0,T}_{\alpha})_{\alpha>0}$.

We consider the Cauchy problem
\begin{equation}
\label{eq5.1} -\frac{\partial u}{\partial t}-L_tu=f_u+\mu,\quad
u(T)=\varphi,
\end{equation}
where $\varphi:E\rightarrow\BR$ is a measurable function such that
$\delta_{\{T\}}\otimes\varphi\cdot m\in\RR(E_{0,T})$,
$\mu\in\RR(E_{0,T})$, $f:[0,T]\times E\times\BR\rightarrow\BR$ is
measurable function and $f_{u}=f(\cdot,\cdot,u)$.

In what follows we maintain the notation of Sections \ref{sec2.2}
and  \ref{sec2.3} concerning time dependent forms and associated
Markov processes. In particular, $\BBM$ is a Markov process
associated with $\EE$ and $\zeta_{\tau}$ is defined by
(\ref{eq2.10}). By abuse of notation, in this section
\[
\langle\mu,u\rangle=\int_{E_{0,T}}u(z)\,\mu(dz)
\]
for $u:E_{0,T}\rightarrow\BR$ and $\mu\in\RR(E_{0,T})$.

We will say that a Borel measurable function $u$ on $E_{0,T}$ is
quasi-c\`adl\`ag  if for q.e. $z\in E_{0,T}$ the process $t\mapsto
u(\BBX_{t})$ is c\`adl\`ag on $[0,T-\tau(0)]$, $P_{z}$-a.s.

\begin{definition}
Let $\delta_{\{T\}}\otimes\varphi\cdot m\in\RR(E_{0,T})$,
$\mu\in\RR(E_{0,T})$. We say that a quasi-c\`adl\`ag function
$u:E_{0,T}\rightarrow\BR$ is a probabilistic solution to
(\ref{eq5.1}) if $f_u\cdot m_1\in\RR(E_{0,T})$ and for q.e. $z\in
E_{0,T}$,
\begin{equation}
\label{eq5.2} u(z)=E_z\Big(\varphi(\BBX_{\zeta_{\tau}})
+\int^{\zeta_{\tau}}_0f_u(\BBX_t)\,dt
+\int^{\zeta_{\tau}}_0dA^{\mu}_t\Big).
\end{equation}
\end{definition}

\begin{remark}
\label{rem4.2} Let $\varphi,f_u,\mu$ be as in the above
definition. Then by  \cite[Proposition 3.4]{K:JFA} the right-hand
side of (\ref{eq5.2}) is a quasi-c\`adl\`ag function of $z$.
Therefore the quasi-c\`adl\`ag requirement on $u$ in the above
definition can be omitted.
\smallskip\\
(ii) From the proof of \cite[Theorem 5.8]{K:JFA} it follows that
if $u$ is a probabilistic solution to (\ref{eq5.1}) then there is
an adapted process $M$ such that $M$ is a martingale under $P_z$
for q.e. $z\in E_{0,T}$ and
\begin{equation}
\label{eq4.10} Y_t=\varphi(\BBX_{\zeta_{\tau}})
+\int^{\zeta_{\tau}}_tf(\BBX_s,Y_s)\,ds
+\int^{\zeta_{\tau}}_tdA^{\mu}_s-\int^{\zeta_{\tau}}_tdM_s,\quad
t\in[0,\zeta_{\tau}],\quad P_z\mbox{-a.s.}
\end{equation}
for q.e. $z\in E_{0,T}$, where $Y_t=u(\BBX_t)$, $t\ge0$.
\end{remark}

\begin{remark}
\label{rem4.3} (i) Let $\nu=\delta_{\{T\}}\otimes\varphi\cdot m$.
If $\varphi\in L^1(E;m)$ then $\nu\in\RR(E_{0,T})$.
\smallskip\\
(ii) One can check that
$A^{\nu}_t=\fch_{[T-\tau(0),\infty]\cap\{T>\tau(0)\}}(t)
\varphi(\BBX_{\zeta_{\tau}})$, $t\ge0$, for $\nu$ defined in (i)
(see the beginning of the proof of \cite[Proposition 3.4]{K:JFA}).
Hence
\[
E_z\varphi(\BBX_{\zeta_{\tau}})=E_z\int_0^{\zeta_{\tau}}dA^{\nu}_t.
\]
%K:\\
%Okreslmy $A$ przez prawa strone wzoru na $A^{\nu}$ i zalozmy, ze
%$\varphi\ge0$. Poniewaz $A$ ma jeden skok w $t=T-\tau(0)$ gdy
%$\tau(0)<T$, to dla $s<T$ i $f:E_{0,T}\rightarrow[0,\infty)$ mamy
%\begin{align*}
%E_{s,x}\int^t_0f(\BBX_{\theta})dA^{\nu}_{\theta}
%&=E_{s,x}\int^t_0f(s+\theta,X_{\zeta\wedge(s+\theta)})\,dA^{\nu}_{\theta}\\
%&=\fch_{[T-s,\infty)}(t)E_{s,x}f(T,X_{\zeta\wedge T})
%\varphi(X_{s+\zeta\wedge(T-s)}),
%\end{align*}
%a dla $s\ge T$ calka powyzsza jest rowna zeru. Dlatego
%\begin{align*}
%E_{m_1}\int^t_0f(\BBX_{\theta})dA^{\nu}_{\theta}&=
%\int_0^T\fch_{[T-s,\infty)}(t)E_{s,m}f(T,X_{\zeta\wedge T})
%\varphi(X_{s+\zeta\wedge(T-s)})\,ds\\
%&=\int_{T-t}^TE_{s,m}f(T,X_{\zeta\wedge T})
%\varphi(X_{s+\zeta\wedge(T-s)})\,ds\\
%&=\int_{T-t}^TE_{s,m}\fch_{\{\zeta>T\}}f(T,X_{T})
%\varphi(X_{T})\,ds\\
%&=\int_{T-t}^TE_{s,m}f(T,X_{T}) \varphi(X_{T})\,ds.
%\end{align*}
%W szczegolnosci, jezeli proces $X$ odpowiada formie symetrycznej,
%to prawa strona jest rowna
%\[
%\int_{T-t}^TE_{s,m}f(T,X_T)\varphi(X_{T})\,ds
%=t\int_Ef(T,y)\varphi(y)\,m(dy)
%\]
%W ogolnosci powinnismy miec zbieznosc
%\[
%\lim_{t\downarrow0}\frac1{t}\int_{T-t}^TE_{s,m}f(T,X_{T})
%\varphi(X_{T})\,ds=E_{T,m}f(T,X_{T})
%\varphi(X_{T})=\int_Ef(T,y)\varphi(y)\,m(dy)
%\]
%(ostatnia rownosc wynika z faktu, ze proces $X$ jest normaly; dla
%zachodzenia pierwszej potrzebujemy na przyklad lewostronnej
%ciaglosci $s\mapsto E_{s,m}f(T,X_{T}) \varphi(X_{T})$). Z drugiej
%strony,
%\[
%\int_{E^1}f(z)\,\nu(dz)=\int_Ef(T,y)\varphi(y)\,m(dy).
%\]
%KK
\end{remark}

Our definition of a renormalized solution is similar to
\cite[Definition 4.1]{PPP}.
\begin{definition}
\label{def4.8} Let $\varphi\in L^1(E;m)$,
$\mu\in\MM_{0,b}(E_{0,T})$. We say that a measurable function
$u:E_{0,T} \rightarrow\BR$ is a renormalized solution of
(\ref{eq5.1}) if
\begin{enumerate}
\item[(a)]$u$ is quasi-c\`adl\`ag, $f_u\in L^1(E_{0,T};m_1)$ and
$T_ku\in\VV_{0,T}$ for every $k>0$,
\item[(b)]there exists a sequence
$\{\lambda_{k}\}\subset\MM_{0,b}(E_{0,T})$ such that
$\|\lambda_{k}\|_{TV}\rightarrow0$ as $k\rightarrow\infty$ and for
every $k\in\BN$  and every bounded $v\in\WW_0$,
\begin{equation}
\label{eq4.15} \EE^{0,T}(T_{k}u,v) =(T_k\varphi,v(T))_{H} +\langle
f_u\cdot m+\mu,\tilde v\rangle +\langle\nu_k,\tilde v\rangle.
\end{equation}
\end{enumerate}
\end{definition}

\begin{theorem}
\label{th4.9} Assume that $\varphi\in L^{1}(E;m)$,
$\mu\in\MM_{0,b}(E_{0,T})$.
\begin{enumerate}
\item[\rm(i)]If  $u$ is a probabilistic solution of
\mbox{\rm(\ref{eq5.1})} and $f_u\in L^1(E_{0,T};m_1)$ then $u$ is
a renormalized solution of \mbox{\rm(\ref{eq5.1})}.
\item[\rm(ii)]If  $u$ is a renormalized
solution of \mbox{\rm(\ref{eq5.1})} then  $u$ is a probabilistic
solution of \mbox{\rm(\ref{eq5.1})}.
\end{enumerate}
\end{theorem}
\begin{proof}
(i) Let $u$ be a probabilistic solution of (\ref{eq5.1}). For
$k>0$ put
\[
Y^k_t=T_ku(\BBX_t),\quad t\ge0.
\]
By Remark \ref{rem4.2} there is a martingale $M$ such that
(\ref{eq4.10}) is satisfied. As in the proof of  Theorem
\ref{th3.5}, applying the Meyer-It\^o formula we show that for
$k>0$,
\begin{equation}
\label{eq4.17} Y^k_t-Y^k_0=\int^t_0\fch_{\{-k<Y_{s-}\le k\}}\,dY_s
-(A^{1,k}_t-A^{2,k}_t)
\end{equation}
for some increasing processes $A^{1,k},A^{2,k}$ such that
\begin{equation}
\label{eq4.32} E_zA^{1,k}_{\zeta_{\tau}}\le
E_z(\varphi-\varphi\wedge k)(\BBX_{\zeta_{\tau}})+
E_z\int^{\zeta_{\tau}}_0\fch_{\{Y_{s-}>k\}}(f_u(\BBX_s)\,ds
+dA^{\mu}_s)
\end{equation}
and
\begin{equation}
\label{eq4.33} E_zA^{2,k}_{\zeta_{\tau}}\le -E_z((\varphi+k)\wedge
0)(\BBX_{\zeta_{\tau}})
-E_z\int^{\zeta_{\tau}}_0\fch_{\{Y_{s-}\le-k\}}(f_u(\BBX_s)\,ds
+dA^{\mu}_s)
\end{equation}
for q.e. $z\in E^1$. Hence $E_zA^{i,k}_{\zeta_\tau}<\infty$ for q.e. $z\in
E_{0,T}$. From this and \cite[Theorem A.3.16]{FOT} it follows that
there is a positive increasing AF $B^{i,k}$ of $\BBM$ such that
$B^{i,k}$ is the compensator of $A^{i,k}$ under $P_z$ for q.e.
$z\in E_{0,T}$. In particular,
\begin{equation}
\label{eq4.38} E_zA^{i,k}_t=E_zB^{i,k}_t,\quad t\ge0,\quad i=1,2
\end{equation}
for q.e. $z\in E_{0,T}$. Since $A^\mu$ is predictable, there
exists a sequence $\{T_n\}$ of predictable stopping times
exhausting the jumps of $A^\mu$, i.e.
\[\{\Delta A\neq 0\}=\bigcup_{n\ge 1}[T_n],
\quad [T_n]\cap [T_m]=\emptyset,\quad n\neq m,
\]
where $[T_n]$ denotes the graph of $T_n$. Let $A^{\mu,c}$ denote
the continuous part of $A^{\mu}$ and let
$A^{\mu,d}=A^{\mu}-A^{\mu,c}$. We have
\begin{align}
\label{eq4.36} E_z\int^{\zeta_{\tau}}_0\fch_{\{Y_{s-}>k\}}\,
dA^{\mu}_s&= E_z\int^{\zeta_{\tau}}_0\fch_{\{Y_{s-}>k\}}\,
dA^{\mu,c}_s+ E_z\int^{\zeta_{\tau}}_0\fch_{\{Y_{s-}>k\}}\,
dA^{\mu,d}_s \nonumber\\
&= E_z\int^{\zeta_{\tau}}_0 \fch_{\{u(\BBX_s)+\Delta
u(\BBX_s)>k\}}\, dA^{\mu,c}_s \nonumber\\
&\quad+ \sum_{n\ge 1}E_z\fch_{\{u(\BBX_{T_n})+\Delta
u(\BBX_{T_n})>k\}}\, \Delta A^\mu_{T_n}.
\end{align}
Since  the filtration $\{\FF_t,t\ge0\}$ is quasi-left continuous
(see \cite[Proposition IV.4.2]{BG}), $\Delta u(X_{T_n})=\Delta
A^\mu_{T_n}$ by Theorem A.3.6 in \cite{FOT}. On the other hand, by
\cite[Theorem 16.8]{GS}, there exists a Borel function
$a:E_{0,T}\rightarrow\BR$ such that $\Delta A^\mu_t=
a(\BBX_{t-})$, $t>0$, $P_z$-a.s. for q.e. $z\in E_{0,T}$. From
this and the fact that $\BBX$ is quasi-left continuous it follows
that $\Delta u(\BBX_{T_n})=a(\BBX_{T_n})$. By this and
(\ref{eq4.36}),
\begin{equation}
\label{eq4.37} E_z\int^{\zeta_{\tau}}_0\fch_{\{Y_{s-}>k\}}\,
dA^{\mu}_s
=E_z\int^{\zeta_{\tau}}_0\fch_{\{u(\BBX_s)+a(\BBX_{s})>k\}}\,dA^\mu_s.
\end{equation}
Analogously to (\ref{eq4.37}) we show that
\begin{equation}
\label{eq4.44} E_z\int^{\zeta_{\tau}}_0\fch_{\{Y_{s-}\le-k\}}\,
dA^{\mu}_s
=E_z\int^{\zeta_{\tau}}_0\fch_{\{u(\BBX_s)+a(\BBX_{s})\le-k\}}\,dA^\mu_s.
\end{equation}
Combining (\ref{eq4.32})--(\ref{eq4.38}) with (\ref{eq4.37}),
(\ref{eq4.44}) we get
\begin{align}
\label{eq4.34} E_zB^{1,k}_{\zeta_{\tau}}&\le
E_z\Big(\fch_{\{\varphi>k\}}\varphi(\BBX_{\zeta_{\tau}})+
\int^{\zeta_{\tau}}_0\fch_{\{u(\BBX_s)>k\}}f_u(\BBX_s)\,ds\nonumber\\
&\qquad+\int^{\zeta_{\tau}}_0\fch_{\{(u+a)(\BBX_s)>k\}}dA^{\mu}_s\Big)
\end{align}
and
\begin{align}
\label{eq4.35} E_zB^{2,k}_{\zeta_{\tau}}&\le
-E_z\Big(\fch_{\{\varphi<-k\}}\varphi(\BBX_{\zeta_{\tau}})+
\int^{\zeta_{\tau}}_0\fch_{\{u(\BBX_s)\le-k\}}f_u(\BBX_s)\,ds\nonumber\\
&\qquad+\int^{\zeta_{\tau}}_0\fch_{\{(u+a)(\BBX_s)\le-k\}}dA^{\mu}_s\Big)
\end{align}
for q.e. $z\in E_{0,T}$. By \cite[Theorem A.3.2]{FOT} the jumps of
$\BBM$ occur in totally inaccessible stopping times, while the
jumps of $B^{i,k}$ in stopping times which are not totally
inaccessible since $B^{i,k}$ is predictable.  Therefore $\BBM$ and
$B^{i,k}$ have no common discontinuities, and hence $B^{i,k}$ is a
positive natural AF of $\BBM$. Let $b^i_k$, $i=1,2$,  denote the
Revuz measure of $B^{i,k}$. From (\ref{eq4.34}), (\ref{eq4.35}),
Remark \ref{rem4.2} and \cite[Proposition 3.13]{K:JFA} we conclude
that
\[
b^1_k(E_{0,T})\le \|\fch_{\{\varphi>k\}}\varphi\|_{L^1(E;m)}
+\|f_u\|_{L^1(E_{0,T};m_1)}+\|\fch_{\{u+a>k\}}\cdot\mu\|_{TV},
\]
\[
b^2_k(E_{0,T})\le \|\fch_{\{\varphi<-k\}}\varphi\|_{L^1(E;m)}
+\|f_u\|_{L^1(E_{0,T};m_1)}+\|\fch_{\{u+a\le-k\}}\cdot\mu\|_{TV}.
\]
Hence $b^1_k,b^2_k\in\MM_{0,b}(E_{0,T})$ and
$\|b^1_k\|_{TV}\rightarrow0$, $\|b^2_k\|_{TV}\rightarrow0$ as
$k\rightarrow\infty$. Combining (\ref{eq4.10}) with (\ref{eq4.17})
we see that
\begin{align*}
Y^{k}_t-Y^{k}_0&=-\int^t_0\fch_{\{-k<Y_{t}\le k\}}
f_u(\BBX_s)\,ds\\
&\quad-\int^t_0\fch_{\{-k<Y_{t-}\le k\}} \,dA^{\mu}_s
-(B^{1,k}_t-B^{2,k}_t)+M^{k}_t
\end{align*}
with
\[
M^k_t=\int^t_0\fch_{\{-k<Y_{s-}\le k\}}\,dM_s
-(A^{1,k}_t-B^{1,k}_t)+A^{2,k}_t-B^{2,k}_t.
\]
By the above and the definition of the measures $b^1_k,b^2_k$ we
have
\begin{align*}
T_{k}u(z)&=E_z\Big(T_k\varphi(\BBX_{\zeta_{\tau}})
+\int^{\zeta_{\tau}}_0\fch_{\{-k<Y_t\le k\}}
f_u(\BBX_t)\,dt \nonumber \\
&\qquad\quad +\int^{\zeta_{\tau}}_0\fch_{\{-k<Y_{t-}\le
k\}}\,dA^{\mu}_t +\int^{\zeta_{\tau}}_0
\,d(A^{b^1_k}_t-A^{b^2_k}_t)\Big).
\end{align*}
From this and (\ref{eq4.37}), (\ref{eq4.44}) we obtain
\begin{equation}
\label{eq5.4} T_{k}u(z) =E_z\Big(T_k\varphi(\BBX_{\zeta_{\tau}})
+\int^{\zeta_{\tau}}_0(f_u(\BBX_t)\,dt+dA^{\mu}_t)
+\int^{\zeta_{\tau}}_0 \,dA^{\nu_{k}}_t \Big)
\end{equation}
with
\[
\nu_k=-\fch_{\{u\notin(-k,k]\}}f_u\cdot
m-\fch_{\{u+a\notin(-k,k]\}}\cdot\mu+b^1_k-b^2_k.
\]
By what has already been proved, $\nu_k\in\MM_{0,b}(E_{0,T})$ and
$\|\nu_k\|_{TV}\rightarrow0$ as $k\rightarrow\infty$. Moreover, by
by \cite[Theorem 3.12]{K:JFA}, $T_{k}u\in\VV_{0,T}$, so what is
left is to show that (\ref{eq4.15}) is satisfied. We shall show
that (\ref{eq4.15}) follows from (\ref{eq5.4}) by the same method
as in elliptic case (see the proof of the fact that (\ref{eq4.1})
follows from (\ref{eq3.20})). Let
$\lambda_k=\delta_{\{T\}}\otimes\varphi\cdot m+f_u\cdot
m+\mu+\nu_k$ and let $A=A^{\lambda_k}$. By Fubini's theorem,
\begin{align*}
R^{0,T}{\lambda_k}(z)-R^{0,T}_{\alpha}\lambda_k(z)
&=E_z\int^{\zeta_{\tau}}_0(1-e^{-\alpha t})\,dA_t\\
&=E_z\int^{\zeta_{\tau}}_0(\int^t_0\alpha e^{-\alpha s}\,ds)\,dA_t
=\alpha E_z\int^{\zeta_{\tau}}_0e^{-\alpha t}
(\int_t^{\zeta_{\tau}}dA_s)\,dt
\end{align*}
for q.e. $z\in E_{0,T}$. Using the definition of $\zeta_{\tau}$
and the fact that $A$ is an AF of $\BBM$ one can check that
$A_{\zeta_{\tau}}-A_t=(A_{\zeta_{\tau}}-A_0)\circ\theta_t$.
%\\
%K:\\
%Na zbiorze $\{T-\tau(0)\ge\zeta\}$ lewa i prawa strona sa rowne
%$A_{\zeta}-A_t$. Dla $\omega$ ze zbioru $\{T-\tau(0)\ge\zeta\}$
%prawa strona wynosi
%\[
%A_{T-\tau(0)(\theta_t\omega)}(\theta_t\omega)-A_0(\theta_t\omega)
%=A_{T-\tau(0)(\theta_t\omega)+t}(\omega)-A_t(\omega)
%=A_{T-\tau(0)(\omega)}(\omega)-A_t(\omega),
%\]
%co jest rowne lewej stronie.\\
%KK\\
Therefore applying the Markov property shows that
\[
R^{0,T}{\lambda_k}(z)-R^{0,T}_{\alpha}\lambda_k(z) =\alpha
E_z\int^{\zeta_{\tau}}_0e^{-\alpha t}
E_{\BBX_t}(\int^{\zeta_{\tau}}_0dA_s)\,dt=\alpha
R^{0,T}_{\alpha}(R^{0,T}\lambda_k)(z)
\]
for q.e. $z\in E_{0,T}$. Since by Remark \ref{rem4.3} and
(\ref{eq5.4}), $T_ku=R^{0,T}\lambda_k$,  it follows from the above
equation that
\begin{equation}
\label{eq4.14} \alpha(T_ku-\alpha
R^{0,T}_{\alpha}(T_ku),v)_{\HH_{0,T}}=\alpha
(R^{0,T}_{\alpha}\lambda_k,v)_{\HH_{0,T}}
\end{equation}
for every bounded $v\in\WW_0$. Since the  left-hand side of
(\ref{eq4.14}) equals  $\EE^{0,T}(\alpha G^{0,T}_{\alpha}T_ku,v)$,
it converges to $\EE^{0,T}(T_ku,v)$ as $\alpha\rightarrow\infty$.
Let $\hat R^{0,T}_{\alpha}$ denote the resolvent associated with
the dual form $\hat\EE^{0,T}$. By \cite[Corollary III.3.8]{S}
applied to the functions  $\alpha\hat R^{0,T}_{\alpha}v$ we may
assume that $\alpha\hat R^{0,T}_{\alpha}v$ converges to $\tilde v$
q.e. as $\alpha\rightarrow\infty$. It follows that the right-hand
side of (\ref{eq4.14}) converges to $\langle\lambda_k,\tilde
v\rangle$ as $\alpha\rightarrow\infty$. Therefore letting
$\alpha\rightarrow\infty$ in (\ref{eq4.14}) we obtain
(\ref{eq4.15}), which completes the proof of (i).
\smallskip\\
(ii) Let $\eta\in L^2(E_{0,T};m_1)$ be a bounded non-negative
function. Then $\hat G^{0,T}\eta\in\WW_0$ and
\[
\EE^{0,T}(T_{k}u,\hat G^{0,T}\eta)=(T_{k}u,\eta)_{\HH_{0,T}}.
\]
From this and (\ref{eq4.15}) it follows that $T_ku$ is a solution
in the sense of duality  (see \cite[Section 4]{K:JFA} for the
definition) of the linear problem
\begin{equation}
\label{eq4.11} (-\frac{\partial}{\partial t}-L_t)T_ku=f_u
+\mu+\nu_k,\quad T_ku(T)=T_k\varphi,
\end{equation}
so by \cite[Corollary 4.2]{K:JFA} $T_ku$ is a probabilistic
solution of the above equation. Therefore (\ref{eq5.4}) (with the
measure $\nu_k$ of (\ref{eq4.11}))  is satisfied. Since
$\|\nu_k\|_{TV}\rightarrow0$ and for every Borel set $F\subset
E_{0,T}$ such that $m(F)<\infty$ we have
\[
(R^{0,T}|\nu_k|,\fch_F)_{\HH_{0,T}}
=\langle|\nu_k|,\widetilde{\hat G^{0,T}\fch_F}\rangle \le T
\|\nu_k\|_{TV},
\]
one can find a subsequence (still denoted by $k$) such that
$R^{0,T}\nu_k(z)\rightarrow0$ as $k\rightarrow\infty$ for
$m_1$-a.e. $z\in E_{0,T}$. Therefore letting $k\rightarrow\infty$
in (\ref{eq5.4}) we show that (\ref{eq5.2}) holds true for
$m_1$-a.e. $z\in E_{0,T}$, and hence for q.e. $z\in E_{0,T}$,
because $u$ and the right-hand side of (\ref{eq5.2}) are
quasi-continuous.
\end{proof}

\begin{remark}
One can show that the function $u+a$ appearing in (\ref{eq4.34})
and (\ref{eq4.35}) is equal quasi everywhere  to the precise
version of $u$ (For the notion of a precise version of a parabolic
potential see \cite{P}).
\end{remark}

We now illustrate the applicability of Theorem \ref{th4.9}. Let us
consider the following hypotheses.

\begin{enumerate}
\item[(P1)]$u\mapsto f(t,x,u)$ is
continuous for every $(t, x)\in E_{0,T}$.
\item[(P2)]There is $\alpha\in\BR$ such that
$(f(t,x,y)-f(t,x,y'))(y-y')\le\alpha|y-y'|^{2}$ for every
$(t,x)\in E_{0,T}$ and $y, y'\in\BR$.
\item[(P3)]$\mu\in \MM_{0,b}(E_{0,T})$ and
$f(\cdot,y)\in L^{1}(E_{0,T};m_1)$ for every $y\in\BR$.
\end{enumerate}

\begin{theorem}
Let  $u_{i}$ be  renormalized solution of \mbox{\rm(\ref{eq5.1})}
with terminal condition $\varphi_{i}$, and right-hand side
$(f^{i},\mu_{i})$, $i=1,2$. If $\varphi_{1}\le\varphi_{2}$
$m_{1}$-a.e., $\mu_{1}\le\mu_{2}$ and either $f^{1}$ satisfies
$\mbox{\rm{(P2)}}$ and $f^{1}_{u_{2}}\le f^{2}_{u_{2}}$
$m_{1}$-a.e. or $f^{2}$ satisfies $\mbox{\rm{(P2)}}$ and
$f^{1}_{u_{1}}\le f^{2}_{ u_{1}}$ $m_{1}$-a.e., then  $u_{1}(z)\le
u_{2}(z)$ for q.e. $z\in E_{0, T}$.
\end{theorem}
\begin{proof}
Follows from Theorem \ref{th4.9} and \cite[Corollary 5.9]{K:JFA}.
\end{proof}

\begin{corollary}
If \mbox{\rm{(P2)}} is satisfied then there exists at most one
renormalized solution of \mbox{\rm{(\ref{eq5.1})}}.
\end{corollary}

\begin{theorem}
Assume \mbox{\rm{(P1)--(P3)}}. Then there exists renormalized
solution of \mbox{\rm{(\ref{eq5.1})}}.
\end{theorem}
\begin{proof}
Follows from Theorem \ref{th4.9} and \cite[Theorem 5.8,
Proposition 5.10]{K:JFA}.
\end{proof}
\medskip

The results of \cite{K:JFA} used in the proof of Theorem
\ref{th4.9} can be generalized to quasi-regular time dependent
Dirichlet forms (see \cite[Remark 4.4]{K:JFA}). Moreover, if the
forms $B^{(t)}$, $t\in[0,T]$, are quasi-regular, then by
\cite[Theorem IV.2.2]{S} there exists a special standard process
$\BBM$ properly associated in the resolvent sense with the time
dependent form determined by $\{B^{(t)},t\in[0,T]\}$. Since one
can check that the results from the theory of stochastic processes
used in the proof of Theorem \ref{th4.9} hold true for such
process $\BBM$,  Theorem \ref{th4.9} can be extended to
quasi-regular time dependent Dirichlet forms.
\bigskip\\
{\large\bf Acknowledgements}
\smallskip\\
This work was supported by NCN Grant No. 2012/07/B/ST1/03508.

\end{document}